\documentclass{amsart}
\usepackage{amsmath,amsthm,amsfonts,amssymb,mathrsfs}
\def\R{{\mathbb R}}
\def\C{{\mathbb C}}

\def\D{{\mathbb D}}
\DeclareMathOperator{\ran}{ran}
\DeclareMathOperator{\Tr}{Tr}

\DeclareMathOperator{\pv}{p.v.}
\newtheorem{theorem}{Theorem}[section]
\newtheorem{proposition}[theorem]{Proposition}
\newtheorem{corollary}[theorem]{Corollary}
\newtheorem{lemma}[theorem]{Lemma}
\newtheorem{remark}{Remark}[section]

\usepackage{graphicx}
\usepackage{subfigure}

\usepackage[latin1]{inputenc}
\usepackage[swedish,english]{babel}

\begin{document}

\title[Spectral bounds for the Neumann-Poincar\'{e} operator]{Spectral bounds for the Neumann-Poincar\'{e} operator on planar domains
with corners}

\author{Karl-Mikael Perfekt}
\address{Centre for Mathematical Sciences, Lund University, 
P.O. Box 118, SE-221 00 Lund, Sweden}
\email{perfekt@maths.lth.se}

\author{Mihai Putinar}
\address{Department of Mathematics, University of California at Santa Barbara,
Santa Barbara, CA 93106-3080}
\email{mputinar@math.ucsb.edu}

\date{\today}

\begin{abstract} 
The boundary double layer potential, or the Neumann-Poincar\'e operator,
is studied on the Sobolev space of order $1/2$ along the boundary,
coinciding with the space of charges giving rise to double layer
potentials with finite energy in the whole space. Poincar\'e's program of
studying the spectrum of the boundary double layer potential is developed
in complete generality, on closed Lipschitz hypersurfaces in Euclidean
space. Furthermore, the Neumann-Poincar\'e operator is realized as a
singular integral transform bearing similarities to the Beurling-Ahlfors
transform in 2D. As an application, bounds for the spectrum of the
Neumann-Poincar\'e operator are derived from recent results in
quasi-conformal mapping theory, in the case of planar curves with corners. 

\end{abstract}
\maketitle

\section{Introduction}  The boundary value problems of potential theory in $\R^n$ are naturally cast as singular integral equations with unknown solutions supported by the boundary. When the boundary is singular, with finite vertex type corners as a mild example or it is a general Lipschitz hypersurface as an extreme example,
notorious analytic complications arise. They were solved with ingenuity and perseverance during a century and a half of continuous struggle by many distinguished mathematicians and physicists. An early account of the fascinating history of
this specific topic of potential theory is offered by the encyclopedia article by Lichtenstein \cite{Licht}.

Singular integrals related to layer potentials are thoroughly studied nowadays, first due to the central role they still hold and the source of inspiration they are 
in modern mathematical analysis, and second due to an array of new applications
to continuum mechanics, field theory and engineering \cite{Ammari}. In this respect, the importance of the spectral analysis of the layer potential integral operators cannot be underestimated. The good news for the working mathematician is that much remains to be done on this front. For instance, the ubiquitous double layer potential integral associated with the Laplacian, also known as the Neumann-Poincar\'e operator, offers a very intriguing spectral picture. To be more specific, let $\Gamma$ be a piecewise smooth planar Jordan curve with finitely many corners, endowed with arclength measure $d\sigma$ and let $K$ denote the Neumann-Poincar\'e operator (see the preliminaries below for the precise definition); recent results of Irina Mitrea \cite{Mitrea} imply a non-real spectrum of $K$ on the Lebesgue spaces $L^p(\Gamma, d\sigma), 1 < p < \infty$, containing closed lemniscate domains, one for each vertex of $\Gamma$. On the other hand,
the spectrum of $K$ on the Sobolev space $H^{1/2}(\Gamma)$, the natural home of charges of double layer potentials
with finite energy, is real, and contained in the interval $(-1,1]$. A notable early contribution goes back to Carleman
\cite{Carleman} who studied in great detail the resolvent of Neumann-Poincar\'e's operator acting on the space of piecewise continuous bounded functions on $\Gamma$. Based on Poincar\'e's original line of thought (see \cite{Khav91} for a modern account of it), one can safely state that considering the action of $K$ on the space $H^{1/2}(\Gamma)$ stands aside as the most natural choice. 

Poincar\'e's program, of estimating the spectrum of $K: H^{1/2}(\Gamma) \longrightarrow H^{1/2}(\Gamma)$
via a Rayleigh quotient (balance) of outer and inner energies of the single layer fields generated by charges
$\rho \in H^{-1/2}(\Gamma)$, can be carried out in great generality. Specifically
we will see its validity on closed Lipschitz hypersurfaces
of $\R^n, n \geq 2$, in Section 3 below. Nonetheless, precise bounds for the spectrum of this operator are very scarce, see for instance \cite{Hels12}.
Two real dimensions are however special, due to the large group of conformal mappings. A groundbreaking observation
due to Ahlfors \cite{Ahl} connects the spectrum of the Neumann-Poincar\'e operator (acting on $H^{1/2}(\Gamma))$ to distortion estimates of quasi-conformal transforms, originally mapping the interior of $\Gamma$
onto the exterior domain. A great deal of work has accumulated on the quasi-conformal estimates \cite{Krus09}, and as incomplete these results may be, they provide the most valuable source of bounds for the spectrum of the Neumann-Poincar\'e operator in 2D. A second, and related, direction of research was advocated in 2D by M. Schiffer, who has obstinately returned during his career to the spatial interpretation of the Neumann-Poincar\'e operator as a Beurling-Ahlfors transform acting on the Bergman space of the inner domain. Naturally, function theory of a complex variable plays a central role in Schiffer's context, see for instance \cite{BS}.

The contents is the following. Section \ref{sec:prelim} develops Poincar\'e's framework of inner and outer harmonic fields, in the case of Lipschitz hypersurfaces $\Gamma$ in $\R^n$. The development is a matter of finding the proper spaces and
formulations of results known for smoother varieties, and hence Section \ref{sec:prelim} is one of technical preliminaries. 

Section \ref{sec:angle} generalizes results of \cite{Khav91} to the Lipschitz boundary case. First by using Plemelj's intertwining formula as the main ingredient in the similarity between the adjoint
$K^*$ of the Neumann-Poincar\'e operator and two different bounded,
self-adjoint operators. These self-adjoint operators are in turn interpreted as angle operators between two canonical orthogonal decompositions of the space of harmonic fields of finite energy. In this way, Schiffer's and Poincar\'e's ideas carry to any number of dimensions in the Lipschitz setting. A simple singular integral operator plays the role of the Beurling-Ahlfors transform, and Poincar\'e's balance of energies is given a direct link to a self-adjoint angle operator. To give a flavor of the results, we state our theorem on the Beurling-Ahlfors transform in 2D.

{
\renewcommand{\thetheorem}{\ref{thm:beurling}}
\begin{theorem}
Let $\Omega \subset \R^2$ be an open and bounded Lipschitz domain with connected boundary and let $T_\Omega : L^2_a(\Omega) \to \overline{L^2_a(\Omega)}$ denote the operator
\begin{equation*}
T_\Omega f (z) = \pv \frac{1}{\pi} \int_\Omega \frac{f(\zeta)}{(\bar{\zeta}-\bar{z})^2} \, dA(\zeta), \quad f\in L^2_a(\Omega), \, z \in \Omega.
\end{equation*}
Then $K^* : H^{-1/2}_0(\partial \Omega) \to H^{-1/2}_0(\partial \Omega)$ is similar to $\overline{T_\Omega} :  L^2_a(\Omega) \to L^2_a(\Omega)$, when the spaces are considered over the field of reals. Here $\overline{T_\Omega}f(z) = \overline{T_\Omega f(z)}$.
\end{theorem}
\addtocounter{theorem}{-1}
} 

In Section 4 we exploit Theorem \ref{thm:beurling} and recent results in quasi-conformal mapping theory to obtain bounds of the spectrum of the operator $K$ in 2D, in the case of domains with corners. An important point to be made is that SOT-type convergence arguments for the singular operator $T_\Omega$ are exceedingly approachable, for example when approximating a domain $\Omega$ with a sequence of domains $\Omega_n$. To exemplify, we give a simple and novel proof of K\"uhnau's angle inequality; if a domain $\Omega$ has a corner of angle $\theta_j$, then the spectral radius of $\overline{T_\Omega}$ satisfies $|\sigma(\overline{T_\Omega})| \geq |1 - \theta_j/\pi|$, see Theorem \ref{thm:kuhnau}. 

In contrast with the spectral properties of $K : L^p \to L^p$ \cite{Mitrea}, any non-essential point in the spectrum of $K : H^{1/2} \to H^{1/2}$ is an isolated eigenvalue of finite multiplicity (see Corollary \ref{cor:spectrum}), owing to the special symmetry features that K exhibits on $H^{1/2}$. In modern computational applications involving $K$, or more general double layer potential operators such as those associated with the Helmholtz equation, such points in the discrete spectrum are easy to recognize. Furthermore, the resolvent can be controlled rather well in their vicinity, especially in comparison to points where there is a continuous contribution to the spectral picture. It is therefore of great interest to obtain information about the essential spectrum of the Neumann-Poincar\'e operator. See for example \cite{Brem12}, \cite{Hels12}. 

To state our main theorem, suppose that $\Omega$ is a curvilinear polygonal domain with interior angles smaller than $\pi$ such that its sum of exterior angles is less than $2\pi$, with one angle regarded as negative. Then by the explicit construction of a conformal mapping from $\Omega$ onto a domain for which the spectral radius of $\overline{T_\Omega}$ is known, we prove the following result on the essential spectrum of the boundary double layer potential.

{
\renewcommand{\thetheorem}{\ref{thm:essbound}}
\begin{theorem} 
Let $\Omega$ be a $C^{1,\alpha}$-smooth curvilinear polygon with $0 < \theta_j < \pi$ for $1 \leq j \leq N$ such that its angles satisfy $$\sum_{j=1}^{N-1}(\pi-\theta_j) + \pi + \theta_N \leq 2\pi,$$
possibly after a cyclic permutation of the vertex labels. Then
\begin{equation*}
|\sigma_\textrm{ess}(K)| = |\sigma_\textrm{ess} (\overline{T_\Omega})| \leq \max_{1 \leq j \leq N} \left( 1-\theta_j/\pi \right),
\end{equation*}
where the Neumann-Poincar\'{e} operator $K$ is acting on $H^{1/2}(\partial \Omega)$.
\end{theorem}
\addtocounter{theorem}{-1}
}

\section{Preliminaries} \label{sec:prelim}
We begin by briefly introducing the Sobolev-Besov spaces and Lipschitz-Hölder spaces that are necessary for our presentation. Let $U \subset \R^n$, $n \geq 2$, be an open Lipschitz domain with connected boundary. For a precise definition of a Lipschitz domain, see for example \cite{Verc84}. $H^1(U)$ denotes the Hilbert space of functions $V \in L^2(U)$ such that
\begin{equation*}
\|V\|_{H^1(U)}^2 = \|V\|_{L^2(U)}^2 + \| \nabla V \|_{L^2(U)}^2 < \infty. 
\end{equation*}
$H^1(\partial U)$ may be instrinsically defined in a similar way, using the Lipschitz manifold structure of $\partial U$. For $0 < s < 1$, we obtain $H^s(\partial U)$ on the real interpolation scale between $L^2(\partial U)$ and $H^1(\partial U)$. Although it will not be used, we note that $H^s(\partial U)$ is a Besov space,
\begin{equation*}
\| v \|_{H^s(\partial U)}^2 \sim \| v \|_{L^2(\partial U)}^2 + \int_{\partial U \times \partial U} \frac{|v(x)-v(y)|^2}{|x-y|^{n-1+2s}} \, d \sigma(x) \, d \sigma(y), 
\end{equation*}
where $\sigma$ denotes $(n-1)$-dimensional Hausdorff measure on $\partial U$. We define $H^{-s}(\partial U)$, $0 \leq s \leq 1$, as the dual of $H^s(\partial U)$ under the (sesquilinear) $L^2$-pairing, and by abuse of notation we always write
\begin{equation*}
g(f) = \langle f, g \rangle_{L^2(\partial U)}.
\end{equation*}
Whenever $1 \in H^{-s}(\partial U)$, $H_0^s(\partial U)$ denotes its annihilator,\begin{equation*}
H_0^s(\partial U) = \{f \in H^s(\partial U) \, : \, \langle f, 1 \rangle_{L^2(\partial U)} = 0 \}, \quad -1 \leq s \leq 1.
\end{equation*}
 For $d \geq 0$ an integer and $0 < \alpha < 1$, we denote by $C^{d,\alpha}(\overline{U})$ the space of d times continuously differentiable functions $\phi : U \to \C$, such that $\partial^{\beta} \phi$ has a continuous extension to $\overline{U}$ satisfying an $\alpha$-H\"{o}lder condition, $\beta = (\beta_1, \ldots , \beta_n)$, $\beta_j \geq 0$, $\sum \beta_j = d$. That is,
\begin{equation*}
|\partial^\beta \phi(x) - \partial^\beta \phi(y)| \lesssim |x-y|^{\alpha}, \quad x,y \in \overline{U}.
\end{equation*}  

Next we recall a few conventions and facts pertaining to layer potentials. Let $\Omega \subset \R^n$, $n \geq 2$, be a bounded Lipschitz domain with connected boundary. Let $G(x,y) = G(x-y,0)$ be the Newtonian kernel, normalized so that $\Delta_x G(x,0) = -\delta$ in the sense of distributions. Explicitly, 
\begin{equation*}
G(x,y) =
\begin{cases}
-\omega_n^{-1} \log |x-y|, & n=2, \\
\omega_n^{-1} |x-y|^{2-n}, & n\geq3,
\end{cases}
\end{equation*}
where $\omega_n$ is the measure of the unit sphere in $\R^n$. By the {\it Neumann-Poincar\'{e} operator}, or the {\it boundary double layer potential}, $K : H^{1/2}(\partial \Omega) \to H^{1/2}(\partial \Omega)$  we mean the operator 
\begin{equation*}
Kf(x) = -2 \int_{\partial \Omega} \partial_{n_y} G(x,y) f(y) \, d \sigma(y), \quad x \in \partial \Omega,
\end{equation*}
where $n_y$ denotes the outward normal derivative at $y$. Here the integral is intended in a principal value sense; its boundedness as an operator on $L^2(\partial \Omega)$ was essentially proven by Coifman, McIntosh and Meyer \cite{Coif82}, see also the classical work of Verchota \cite{Verc84}. The boundedness of $K$ on $H^{1/2}(\partial \Omega)$ follows by exploiting symmetry features of $K$ and interpolation, see for example Lemma \ref{lem:Krange}. By $K^*$ we always mean the adjoint with respect to the $L^2$-pairing, so that, in our setting, $K^*$ is an operator acting on $H^{-1/2}(\partial \Omega)$.

For $x \notin \partial \Omega$ we write
\begin{equation*}
Df(x) =  \int_{\partial \Omega}  \partial_{n_y}  G(x,y) f(y) \, d \sigma(y), \quad x \notin \partial \Omega,
\end{equation*}
and call $D$ the {\it double layer potential}. For $g \in H^{-1/2}(\partial \Omega)$ the single layer potential $S$ is defined by
\begin{equation*}
Sg(x) = \int_{\partial \Omega} G(x,y) g(y) \, d \sigma(y), \quad x \in \R^n.
\end{equation*}

Note that the operator $S : L^2(\partial \Omega) \to L^2(\partial \Omega)$ is self-adjoint and bounded; furthermore we may assume that $\Omega$ is a {\it normal domain} in the sense that $S : L^2(\partial \Omega) \to H^1(\partial \Omega)$ is a bijective bounded operator. By \cite{Verc84}, this is equivalent to the existence of a function $g_0 \in L^2(\partial \Omega)$ such that $S g_0 |_{\overline{\Omega}} \equiv 1$. Note that $\ker(I-K^*) = \C g_0$ (see Lemma \ref{lem:Krange}). By a duality and interpolation argument, the extension $S : H^{-1/2}(\partial \Omega) \to H^{1/2}(\partial \Omega)$ is then also continuous and invertible. If $n \geq 3$ every domain is normal, but when $n=2$ there exist domains which are non-normal. However, if $\Omega$ is non-normal, then
$$t\Omega = \{tx \, : \, x \in \Omega\}$$
is normal (see \cite{Verc84}) for every $t > 0$, $t \neq 1$. Since $K_\Omega$ and $K_{t\Omega}$ are unitarily equivalent and all our main results are invariant under such homotheties, we henceforth assume that $\Omega$ is normal.

Denote the exterior of $\Omega$ by $\Omega_e = \overline{\Omega}^c$ and by $\mathfrak{H}$ the space of harmonic functions $h$ on $\Omega \cup \Omega_e$ with $\lim_{x \to \infty} h(x) = 0$ and finite energy,
\begin{equation*}
\|h \|_{\mathfrak{H}} = \int_{\Omega \cup \Omega_e} |\nabla h|^2 \, dx < \infty.
\end{equation*}
To ensure that $\mathfrak{H}$ is a Hilbert space we also require that if $h \neq 0$ and $h_e = h|_{\Omega_e} = 0$, then $h_i = h|_\Omega$ is non-constant.

Each element $h \in \mathfrak{H}$ has an interior trace $\Tr_{\mathrm{int}} h = \Tr h_i \in H^{1/2}(\partial \Omega)$ and an exterior trace $\Tr_{\mathrm{ext}} h = \Tr h_e \in H^{1/2}(\partial \Omega)$. By the classical Poincar\'{e} inequality for bounded Lipschitz domains $U$ and the fact that the trace $\Tr : H^1(U) \to H^{1/2}(\partial U)$ is continuous, we see that the interior and the exterior traces are continuous as maps from $\mathfrak{H}$ to $H^{1/2}(\partial \Omega)$. 

The trace normal derivatives $\partial_n^{\textrm{int}}h, \partial_n^{\textrm{ext}} h \in H^{-1/2}(\partial \Omega)$ are defined by duality and via Green's formula. To be more specific,
\begin{equation*}
\langle \partial_n^{\textrm{int}} h, v \rangle_{L^2(\partial \Omega)} = \int_{\Omega} \langle \nabla h, \nabla V_i \rangle \, dx, \quad \langle \partial_n^{\textrm{ext}} h, v \rangle_{L^2(\partial \Omega)} = -\int_{\Omega_e} \langle \nabla h, \nabla V_e \rangle \, dx,
\end{equation*}
for $v \in H^{1/2}(\partial \Omega)$, $V_i \in H^1(\Omega)$ with $\Tr V_i = v$ and $V_e \in H^1(\Omega_e)$ with $\Tr V_e = v$. These definitions are meaningful. Suppose for example that $\Tr V_i = 0$. Then there exist functions $V_n \in C_c^\infty(\Omega)$ such that $V_n \to V_i$ in $H^1(\Omega)$ \cite{Mikh08} and a sequence of $C^\infty$ domains $\Omega_j$, compactly contained in $\Omega$ and converging to $\Omega$ in an appropriate sense \cite{Verc84}. Then
\begin{equation*}
\int_{\Omega} \langle \nabla h, \nabla V_n \rangle \, dx = \lim_{j\to\infty} \int_{\Omega_j} \langle \nabla h, \nabla V_n \rangle \, dx = \lim_{j\to\infty} \langle \partial_n h, V_n \rangle_{L^2(\partial \Omega_j)} = 0.
\end{equation*}
Letting $n \to \infty$ we obtain that $\int_{\Omega} \langle \nabla h, \nabla V_i \rangle \, dx = 0$, as desired. A similar argument shows that $\partial_n^{\textrm{ext}} h$ is well-defined. On several occasions when we use Green's formula on Lipschitz domains in this paper, an argument like the one just presented is implicit. Note also that as maps, $\partial_n^{\textrm{int}} : \mathfrak{H} \to H_0^{-1/2}(\partial \Omega)$ and $\partial_n^{\textrm{ext}} : \mathfrak{H} \to H^{-1/2}(\partial \Omega)$ are continuous.

One interprets an element $h \in \mathfrak{H}$ as a pair $(h_i, h_e) = (h|_\Omega, h|_{\Omega_e})$, with the corresponding orthogonal decomposition $\mathfrak{H} = \mathfrak{H}_i \oplus \mathfrak{H}_e$. We denote by $P_i$ and $P_e$ the orthogonal projections onto $\mathfrak{H}_i$ and $\mathfrak{H}_e$ respectively, so that $P_i (h_i,h_e) = (h_i, 0)$ and $P_ e(h_i,h_e) = (0, h_e)$. 

Another natural orthogonal decomposition of $\mathfrak{H}$ holds, distinguishing among the single and double layer potentials. Specifically, let $$\mathfrak{S} = \{h \in \mathfrak{H} \, : \, \Tr_{\mathrm{int}}h = \Tr_{\mathrm{ext}} h \}$$ denote the space of single layer potentials in $\mathfrak{H}$, and let $$\mathfrak{D} = \{h \in \mathfrak{H} \, : \, \partial_n^{\mathrm{int}}h = \partial_n^{\mathrm{ext}} h\}$$ denote the space of double layer potentials. Then $\mathfrak{H} = \mathfrak{S} \oplus \mathfrak{D}$ and we write $P_s$ and $P_d$ for the corresponding projections. As expected we have that $S : H^{-1/2}(\partial \Omega) \to \mathfrak{H}$ is continuous and $\mathfrak{S} = S(H^{-1/2}(\partial \Omega))$ if $n \geq 3$,  while $S : H_0^{-1/2}(\partial \Omega) \to \mathfrak{H}$ is continuous and $\mathfrak{S} = S(H_0^{-1/2}(\partial \Omega))$ when $n=2$. Similarly, $D : H_0^{1/2}(\partial \Omega) \to \mathfrak{H}$ is continuous and $\mathfrak{D} = D(H_0^{1/2}(\partial \Omega))$ for all $n \geq 2$. Note that $D(\C) = \C (1,0)$ is the subspace we subtracted from $\mathfrak{H}$ in order to have a norm. See Section 1 of \cite{Khav91} for proofs which carry over verbatim to our Lipschitz setting for all the statements of this paragraph.

We remark here that Poincar\'{e}'s inequality does not necessarily hold for the exterior domain, so that for an element $h \in \mathfrak{H}$, $h_e$ is not necessarily in $H^1(\Omega_e)$. In spite of this, the exterior Dirichlet problem is well posed in the sense that $\Tr_{\mathrm{ext}} : \mathfrak{H}_e \to H^{1/2}(\partial \Omega)$ is continuous and invertible, where $\mathfrak{H}_e = P_e \mathfrak{H}$. We have already seen that $\Tr_{\mathrm{ext}}$ is continuous, and its surjectivity follows by considering single layer potentials. It remains to check the injectivity. Suppose that $\Tr_{\mathrm{ext}} h_e = 0$. Then, for any $V \in \mathfrak{H}_e$ it follows that
\begin{equation*}
\langle \partial_n^{\mathrm{ext}} h_e , \Tr_{\mathrm{ext}} V \rangle_{L^2(\partial \Omega)} = \langle \Tr_{\mathrm{ext}} h_e , \partial_n^{\mathrm{ext}} V \rangle_{L^2(\partial \Omega)} = 0,
\end{equation*}
and hence that also $\partial_n^{\mathrm{ext}} h_e = 0$. But then $(0,h_e)$ represents a function harmonic on all of $\R^n$ and hence $h_e = 0$. For if $\phi \in C_c^\infty(\R^n)$, then
\begin{equation*}
\int_{\Omega_e} h_e \Delta \overline{\phi} \, dx = - \int_{\Omega_e} \langle \nabla h_e, \nabla \phi \rangle \, dx = \langle \partial_n^{\mathrm{ext}} h_e , \phi \rangle_{L^2(\partial \Omega)} = 0,
\end{equation*}
which implies that $(0,h_e)$ is harmonic across $\partial \Omega$. This supplies the details for the terse argument presented in (\cite{Hels12}, Proposition 4.1).

For $f \in H^{-1/2}(\partial \Omega)$ ($f \perp 1$ if $n=2$) and $g \in H_0^{1/2}(\partial \Omega)$, arguing with smooth functions and the continuity of operators involved, the well known jump formulae \cite{Verc84} for $S$ and $K$ take on the form
\begin{align*}
\Tr_{\mathrm{int}} Sf& = \Tr_{\mathrm{ext}} Sf = Sf|_{\partial \Omega},& \partial_n^{\textrm{int}} Sf& = \frac{1}{2}(f-K^*f), \\
\partial_n^{\textrm{ext}} Sf& = \frac{1}{2}(-f-K^*f),& \Tr_{\mathrm{int}} Dg& = \frac{1}{2}(-g-Kg), \\
\Tr_{\mathrm{ext}} Dg& = \frac{1}{2}(g-Kg), & \partial_n^{\textrm{int}}Dg& = \partial_n^{\textrm{ext}}Dg.
\end{align*}

\section{The angle operators} \label{sec:angle}
The present section is aimed at interpreting, up to similarities, the operator $K^* : H_0^{-1/2}(\partial \Omega) \to H_0^{-1/2}(\partial \Omega)$ as two different angle operators between the two orthogonal decompositions of the Hilbert space 
$\mathfrak H$. The first of these operators will be given a concrete realization as a generalized Beurling-Ahlfors singular integral transform, while the second will put Poincar\'{e}'s Rayleigh quotient of energies into its correct light in the case of a Lipschitz domain. The arguments essentially follow those of \cite{Khav91}, with some additional technicalities, addressed in \cite{Hels12}, arising from the fact that $K$ is no longer a compact operator on the scale $H^s$ of Besov spaces, $0 \leq s \leq 1$.

\begin{lemma} \label{lem:Krange}
For $0 \leq s \leq 1$ denote by $W^s \subset H^s(\partial \Omega)$ the subspace
\begin{equation*}
W^s = \{f \in H^s \, : \, \langle f, g_0 \rangle_{L^2} = 0 \}.
\end{equation*}
Then $I-K:W^{s} \to W^{s}$ is invertible. In particular, for $I-K:H^{s} \to H^{s}$ we have $\ran(I-K) = W^{s}$.
\end{lemma}
\begin{proof}
By \cite{Verc84}, $I - K : W^0 \to W^0$ and $I-K^* : L^2_0 \to L^2_0$ are invertible. Consider the single layer potential as an invertible operator $S : L^2 \to H^1$. A simple computation shows that $SL^2_0 = W^1$. From Plemelj's formula $KS = SK^*$ \cite{Chang08} we find
\begin{equation*}
I - K = S(I-K^*)S^{-1},
\end{equation*}
so that also $I-K : W^1 \to W^1$ is invertible. By a real interpolation argument it follows that $I-K:W^{s} \to W^{s}$ is invertible. The last statement is now obvious since $(I-K^*)g_0 = 0$.

\end{proof}
Suppose that $g \in H_0^{-1/2}(\partial \Omega)$, or equivalently that $Sg = -\frac{1}{2}(I-K)f$ for some unique $f  \in H_0^{1/2}(\partial \Omega)$, by the previous lemma and the fact that $\ker(I-K) = \C$. By the jump formulae this equality means precisely that $Df + Sg = 0$ in $\Omega_e$, which leads to the computation

\begin{align}\label{eq:unitaryeq1} \notag
\|Df + Sg\|_\mathfrak{H}^2 &= \int_{\partial \Omega} \left( -\frac{1}{2}(f+Kf) + Sg \right) \overline{ \partial_n^{\textrm{int}} (Df + Sg) } \, d\sigma \\ &= -\int_{\partial \Omega} f \overline{ \left( \partial_n^{\textrm{int}} (Df + Sg) -  \partial_n^{\textrm{ext}} (Df + Sg) \right)} \, d\sigma = -\int_{\partial \Omega} f \overline{ g} \, d\sigma.
\end{align}

Summing up, $J : H_0^{-1/2} \to H_0^{1/2}$, $J = 2(I-K)^{-1}S$ is a strictly positive and bijective continuous operator. For future reference we also note that a similar computation shows that
\begin{equation} \label{eq:unitaryeq2}
\langle Df - Sg, Df+Sg \rangle_\mathfrak{H} = -\int_{\partial \Omega} Kf \overline{g} \, d\sigma = \langle KJg, g \rangle_{L^2(\partial \Omega)}. 
\end{equation}
As an operator considered on $L^2_0$, $J : L^2_0 \to L^2_0$ is of course still strictly positive with dense range and hence has an injective square root $\sqrt{J}$ with dense range.
\begin{lemma} \label{lem:sqrtj}
The operator $\sqrt{J}$ extends to a bicontinuous bijection $\sqrt{J} : H_0^{-1/2} \to L^2_0$. 
\end{lemma}
\begin{proof}
Noting that $(g,f) \mapsto \langle Jg, f \rangle_{L^2}^{1/2}$ is a scalar product for $g \in H_0^{-1/2}$ and $f  \in H_0^{1/2}$  Cauchy-Schwarz inequality implies
\begin{equation*}
 |\langle g, f \rangle_{L^2}|^2 = |\langle Jg, J^{-1}f \rangle_{L^2}|^2 \leq \langle Jg, g \rangle_{L^2} \langle f, J^{-1}f \rangle_{L^2} \lesssim \| \sqrt{J}g \|_{L^2}^2 \|f\|_{H^{1/2}}^2,
\end{equation*}
at least for $g \in L^2_0$. Letting $\ell(f) = \int_{\partial \Omega} f \, d\sigma$,  this implies for $f \in H^{1/2}$ that
\begin{equation*}
 |\langle g, f \rangle_{L^2}| = |\langle g, f - \ell(f) \rangle_{L^2}| \lesssim \| \sqrt{J}g \|_{L^2} \|f - \ell(f) \|_{H^{1/2}} \lesssim \| \sqrt{J}g \|_{L^2} \|f\|_{H^{1/2}}
\end{equation*}
 We obtain for sufficiently smooth $g$,
\begin{equation} \label{eq:bddbelow}
\|g\|_{H^{-1/2}} \lesssim \| \sqrt{J}g \|_{L^2}.
\end{equation}
For $h \in L^2_0$ one has the chain of inequalities 
\begin{equation*}
|\langle \sqrt{J} g, \sqrt{J} h \rangle_{L^2}| \leq  \|g\|_{H^{-1/2}} \| Jh \|_{H^{1/2}}   \lesssim \|g\|_{H^{-1/2}}  \| h \|_{H^{-1/2}} \lesssim \|g\|_{H^{-1/2}} \| \sqrt{J}h \|_{L^2}.
\end{equation*}
Due to the to the denseness of the range of $\sqrt{J}$ in $L^2_0$ and a duality argument like above, one finds that $\sqrt{J}$ extends to a continuous operator $\sqrt{J} : H_0^{-1/2} \to L^2_0$. We have seen in \eqref{eq:bddbelow} that this operator is bounded from below, and since it also has dense range, the proof is complete.
\end{proof}
Consider for a moment the operator $K_1 : L^2_0 \to L^2_0$ given by $K_1 = P_{L^2_0}K $. Then, regarding the adjoint as an operator $K_1^* : L^2_0 \to L^2_0$ we infer $K_1^* = K^*|_{L^2_0}$ and
\begin{equation*}
K_1J = JK_1^*,
\end{equation*}
the latter equation owing to Plemelj's symmetrization principle $KS = SK^*$. Denoting by $A : L^2_0 \to L^2_0$  the operator $A = \sqrt{J}K^*\sqrt{J}^{-1}$, it follows that $A$ is self-adjoint and that $K^* : H^{-1/2}_0 \to H^{-1/2}_0$ is similar to $A$.

For $g \in H_0^{-1/2}$, let $h = -DJg + Sg \in \mathfrak{H}_i$ as before, $$\mathfrak{H}_i = \{h \in \mathfrak{H} \, : \, h = (h_i,0)\}.$$ In the present language, \eqref{eq:unitaryeq1} and \eqref{eq:unitaryeq2} express that $\| \sqrt{J}g \|_{L^2_0} = \|h\|_{\mathfrak{H}}$ and 
\begin{equation*}
\langle (P_d - P_s)h, h \rangle_\mathfrak{H} = \langle A \sqrt{J}g, \sqrt{J}g \rangle_{L^2_0}.
\end{equation*}
Hence $\sqrt{J}g \mapsto h$ is a unitary map of $L^2_0$ onto $\mathfrak{H}_i$, giving that $A$ is unitarily equivalent to the operator $P_i(P_d - P_s)P_i : \mathfrak{H}_i \to \mathfrak{H}_i$, since both operators are self-adjoint. The preceding computations and observations are summarized in the next proposition.
\begin{proposition}
The operator $K^*$ (acting on $H^{-1/2}_0(\partial \Omega)$) is similar to the self-adjoint operator $A = \sqrt{J}K^*\sqrt{J}^{-1}$ (acting on $L^2_0(\partial \Omega))$. Furthermore, $A$ is unitarily equivalent to the angle operator $P_i(P_d - P_s)P_i$ acting on $\mathfrak{H}_i = \{h \in \mathfrak{H} \, : \, h = (h_i,0)\}$.
\end{proposition}
Following the computations of \cite{Khav91} in our Lipschitz setting we may now realize the angle operator $P_i(P_d - P_s)P_i$ as an operator acting on the $\R^n$-valued Bergman type space $\mathfrak{B}(\Omega)$,
\begin{equation*}
\mathfrak{B}(\Omega) = \{\nabla u \in L^2(\Omega) \, : \, \Delta u = 0\}.
\end{equation*}
Defining the operator
\begin{equation*}
\Pi_\Omega(\nabla u)(x) = \pv \nabla_x \int_{\Omega} \nabla_y G(x,y) \cdot \nabla_y u \, dy, \quad x \in \Omega,
\end{equation*}
we have the following result.
\begin{theorem} \label{thm:angleop}
Let $\Omega \subset \R^n$ be an open and bounded Lipschitz domain with connected boundary. The angle operator $P_i(P_d - P_s)P_i : \mathfrak{H}_i \to \mathfrak{H}_i$ is unitarily equivalent to the operator $B_\Omega : \mathfrak{B}(\Omega) \to \mathfrak{B}(\Omega),$ 
\begin{equation*}
B_\Omega = I + 2\Pi_\Omega.
\end{equation*}
In addition, the operator $K^* : H^{-1/2}_0(\partial \Omega) \to H^{-1/2}_0(\partial \Omega)$ is similar to $B_\Omega$.
\end{theorem}
\begin{proof}
Repeating the calculations of (\cite{Khav91}, Lemma 5) verbatim, we have for $f \in H_0^{1/2}(\partial \Omega)$ and $g \in H_0^{-1/2}(\partial \Omega)$ that
\begin{equation*}
\pv \int_{\R^n} \nabla_y G(x,y) \cdot \nabla_y Sg(y) \, dy = -Sg(x), \quad x \in \Omega,
\end{equation*}
and
\begin{equation*}
\pv \int_{\R^n} \nabla_y G(x,y) \cdot \nabla_y Df(y) \, dy = 0, \quad x \in \Omega.
\end{equation*}
Considering the unique decomposition $h = Df + Sg$ of any $h \in \mathfrak{H}_i$, the result follows via the unitary identification $\mathfrak{H}_i \ni (h_i,0) \mapsto \nabla h_i \in B(\Omega)$.
\end{proof}
Studying $K^*$ acting on $H^{-1/2}_0$, rather than on $H^{-1/2}$, only eliminates the point $1$ from the spectrum of $K^*$.
\begin{corollary} \label{cor:spectrum}
The spectrum $\sigma(K)$ of $K : H^{1/2} \to H^{1/2}$ coincides with that of $B_\Omega$, except for the point 1. The essential spectra also coincide, 
\begin{equation*}
\sigma_{\textrm{ess}}(K) = \sigma_{\textrm{ess}}(B_\Omega).
\end{equation*}
Furthermore, $\sigma(K) \subset \R$ and any point $\lambda \in \sigma(K) \setminus \sigma_{\textrm{ess}}(K)$ is an isolated eigenvalue of finite multiplicity, since $B_\Omega$ is self-adjoint.
\end{corollary}
\begin{remark} \rm
Chang and Lee \cite{Chang08} prove that $\sigma(K) \subset (-1,1]$. See also \cite{Hels12}.
\end{remark}
In two dimensions, $n=2$, we may formulate Theorem \ref{thm:angleop} in terms of the usual Bergman space $L^2_a(\Omega)$ of analytic functions, 
\begin{equation*}
L^2_a(\Omega) = L^2(\Omega) \cap \textrm{Hol}(\D),
\end{equation*}
its anti-analytic counterpart $\overline{L^2_a(\Omega)}$, and the Beurling-Ahlfors transform. 

To make this formulation precise, we must consider $\mathfrak{H}$ and $H_0^{-1/2}(\partial \Omega)$ to consist only of real-valued functions. Then we may identify $B(\Omega)$ with $\overline{L^2_a(\Omega)}$, since any $F \in \overline{L^2_a(\Omega)}$ is of the form $F = \nabla h_i = 2\bar{\partial} h_i$ for some $h \in \mathfrak{H}$. Note however that we are considering $\overline{L^2_a(\Omega)}$ as a Hilbert space over the reals. 

Let $h = Df + Sg \in \mathfrak{H}_i$ for real-valued $f \in H^{1/2}_0$ and $g \in H^{-1/2}_0$. Then a computation similar to that in Theorem \ref{thm:angleop} shows that
\begin{equation*}
\pv \frac{1}{\pi} \int_\Omega \frac{\partial(Df + Sg)}{(\bar{\zeta}-\bar{z})^2} \, dA(\zeta) = \bar{\partial}(Df - Sg)(z), \quad z \in \Omega,
\end{equation*}
where $dA$ denotes area measure. Note in particular that 
\begin{equation*}
\nabla_\zeta \nabla_z G(\zeta,z) = \frac{1}{\pi} \frac{1}{(\bar{\zeta}-\bar{z})^2}.
\end{equation*}
In conclusion, Theorem \ref{thm:angleop} can be restated as follows. 
\begin{theorem} \label{thm:beurling}
Let $\Omega \subset \R^2$ be an open and bounded Lipschitz domain with connected boundary and let $T_\Omega : L^2_a(\Omega) \to \overline{L^2_a(\Omega)}$ denote the operator
\begin{equation*}
T_\Omega f (z) = \pv \frac{1}{\pi} \int_\Omega \frac{f(\zeta)}{(\bar{\zeta}-\bar{z})^2} \, dA(\zeta), \quad f\in L^2_a(\Omega), \, z \in \Omega.
\end{equation*}
Then $K^* : H^{-1/2}_0(\partial \Omega) \to H^{-1/2}_0(\partial \Omega)$ is similar to $\overline{T_\Omega} :  L^2_a(\Omega) \to L^2_a(\Omega)$, when the spaces are considered over the field of reals. Here $\overline{T_\Omega}f(z) = \overline{T_\Omega f(z)}$.
\end{theorem}
Note that the operator $T_\Omega$ is defined regardless of topological assumptions on $\Omega$ such as boundedness and smoothness. It is also straightforward to check that if $L$ is a fractional linear transformation, then $\overline{T_\Omega}$ and $\overline{T_{L(\Omega)}}$ are unitarily equivalent. As in \cite{Khav91}, we remark the following symmetries. Their proofs remain the same, except having to work with approximate eigenvalues rather than eigenvalues, noting that the spectrum of the symmetric operator $\overline{T_\Omega}$ is equal to its set of approximate eigenvalues.
\begin{corollary}
Let $\Omega \subset \R^2$ be an open and bounded Lipschitz domain with connected boundary. Excepting the point $1$, the spectrum $\sigma(K)$ of $K : H^{1/2}(\partial \Omega) \to H^{1/2}(\partial \Omega)$ is symmetric with respect to the origin. 
\end{corollary}
\begin{corollary} \label{cor:symmetry}
Let $\Omega \subset \R^2$ be an open and bounded Lipschitz domain with connected boundary. Then $\sigma(\overline{T_\Omega}) = \sigma(\overline{T_{\Omega_e}})$.
\end{corollary}

As the final matter of this section we shall consider an alternative angle operator, $P_s(P_e-P_i)P_s : \mathfrak{S} \to \mathfrak{S}$. It seems geometrically plausible that this operator should have strong ties to the angle operator $P_i(P_d-P_s)P_i$, and hence to the Neumann-Poincar\'{e} operator. In fact, $P_s(P_e-P_i)P_s$ is intimately connected to Poincar\'{e}'s variational problem of inner and outer energies. To explain our point of view, let us briefly recall and combine some of the ideas of \cite{Hels12} and \cite{Khav91}. The analogies with the previous work of this section should be clear.

Let us restrict ourselves to $n=2$ in our discussion. The higher dimensional cases are simpler. We will let $\sqrt{S} : H^{-1/2}_0 \to W^{0}$ play the role that $\sqrt{J}$ had earlier.  Note that $S: H^{-1/2}_0(\partial \Omega) \to W^{1/2}$ is an invertible positive operator, so that $\sqrt{S}$ is invertible (cf. Lemma \ref{lem:sqrtj}). Define an operator $B : W^0 \to W^0$ by $B = \sqrt{S} K^* \sqrt{S}^{-1}$. From Plemelj's formula $KS = SK^*$ we see that $B$ is self-adjoint, and by construction it is similar to $K^* : H^{-1/2}_0 \to H^{-1/2}_0$. Furthermore, one easily verifies that for $g \in H^{-1/2}_0(\partial \Omega)$
\begin{align*}
\|Sg\|^2_{\mathfrak{H}} &= \langle \sqrt{S}g, \sqrt{S}g \rangle_{L^2(\partial \Omega)}, \\
\langle (P_e - P_i)Sg, Sg \rangle_{\mathfrak{H}} &= \langle B \sqrt{S}g, \sqrt{S}g \rangle_{L^2(\partial \Omega)}.
\end{align*}
We conclude that $B$ is unitarily equivalent to $P_s(P_e-P_i)P_s : \mathfrak{S} \to \mathfrak{S}$. In particular, the spectral radius of $K^*$ may be computed via the numerical range of this operator,
\begin{equation} \label{eq:poincare}
\left|\sigma \left( K^*|_{H^{-1/2}_0} \right)\right| = \sup_{g \in H_0^{-\frac{1}{2}}} \frac{\| \nabla Sg \|^2_{L^2(\Omega_e)} - \| \nabla Sg \|^2_{L^2(\Omega)}}{\| \nabla Sg \|^2_{L^2(\Omega_e)} + \| \nabla Sg \|^2_{L^2(\Omega)}}, \quad n=2,
\end{equation}
which is exactly the supremum of Poincar\'{e}'s quotient. We state our conclusion in a theorem.
\begin{theorem}
Let $\Omega \subset \R^n$, $n \geq 2$, be an open and bounded Lipschitz domain with connected boundary. The operator $K^* : H^{-1/2}(\partial \Omega) \to H^{-1/2}(\partial \Omega)$ ($K^* : H^{-1/2}_0 \to H^{-1/2}_0$ if $n =2$) is similar to the angle operator $P_s(P_e-P_i)P_s : \mathfrak{S} \to \mathfrak{S}$ acting on the space of single layer potentials.
\end{theorem}

\section{Spectral bounds for planar domains} \label{sec:bounds}
The points $$\lambda \in \sigma(K) \setminus \{1\} = \sigma \left(K^*|_{H^{-1/2}_0}\right) = \sigma(\overline{T_\Omega})$$ are known as the Fredholm eigenvalues of $\Omega$, in the case that $\Omega$ is a smooth bounded planar domain. The largest eigenvalue is often of interest. Theorem \ref{thm:beurling} allows us to define the largest Fredholm eigenvalue as $|\sigma(\overline{T_\Omega})|$ for any simply connected domain $\Omega$ whose boundary is given by a closed Lipschitz curve in $C^* = \C \cup \{\infty\}$. The most common definition of the largest Fredholm eigenvalue for a non-smooth domain is given by Poincar\'{e}'s extremal problem \eqref{eq:poincare}, see \cite{Krus09}. Since they coincide for bounded Lipschitz domains by the results in Section \ref{sec:angle}, and both definitions are invariant under fractional linear transformations, Schober \cite{Schob73}, these two notions of a largest Fredholm eigenvalue coincide. We remark that the largest Fredholm eigenvalue is given by the numerical range of the symmetric operator $\overline{T_\Omega}$,
\begin{equation} \label{eq:numrange}
|\sigma(\overline{T_\Omega})| = \|\overline{T_\Omega}\| = \sup_{\|f\|_{L^2_a(\Omega)}=1} \langle \overline{T_\Omega}f, f \rangle_{L^2_a(\Omega)}.
\end{equation}

The largest Fredholm eigenvalue has been determined explicitly using quasiconformal mapping techniques for certain domains, for example for regular polygons and rectangles sufficiently close to a square, see K\"uhnau \cite{Kuhn88} and Werner \cite{Wern97}. We shall make use of the following result of Krushkal \cite{Krus09}, proven by methods of holomorphic motion.
\begin{theorem}{\cite{Krus09}} \label{thm:krushkal}
Let $\Omega \subsetneq \R^2$ be an unbounded convex domain with piecewise $C^{1,\alpha}$-smooth boundary, $\alpha > 0$. Denote by $0 < \theta < \pi$ the least interior angle made between the boundary arcs of $\partial \Omega$, taking into consideration also the angle made at $\infty$. Then the largest Fredholm eigenvalue of $\Omega$ is $1 - \theta/\pi$.
\end{theorem}
\begin{remark} \rm
Often the angle at infinity is considered to have the negative sign. For the purpose of this paper we will, however, consider it to be a positive angle. Despite this, it is admittedly instructive to think of it as a negative angle in some of the upcoming results.
\end{remark}
Denote by $\mathcal{V}$  the class of all domains $\Omega$ described by the above theorem, as well as their fractional linear images $L(\Omega)$. Hence, for every $V \in \mathcal{V}$ we know that
\begin{equation*}
|\sigma(\overline{T_V})| = \max_{1 \leq j \leq N} \left( 1-\theta_j/\pi \right),
\end{equation*}
where $\theta_1, \ldots, \theta_N$ are the $N$ interior angles of $V$. 

By a $C^{1,\alpha}$-smooth curvilinear polygon $\Omega \subset \R^2$, $0 < \alpha < 1$, we shall mean an open, bounded and simply connected set whose boundary is curvilinear polygonal in the following sense: there are a finite number of counter-clockwise consecutive vertices $(a_j)_{j=1}^N \subset \R^2$, $1 \leq N < \infty$, and $C^{1,\alpha}$-smooth arcs $\gamma_j : [0,1] \to \C$ with starting point $a_j$ and end point $a_{j+1}$ such that $\partial \Omega = \cup_{j} \gamma_j$ and $\gamma_j$ and $\gamma_{j+1}$ meet at an interior angle $0 < \theta_{j+1} < 2\pi$ at the point $a_{j+1}$. In this definition indices are to be understood modulo $N$. 

One major advantage of working with the 2D singular integral operator $T_\Omega$ is the easy accessibility to  convergence arguments of SOT-type, for example when approximating a domain $\Omega$ with better domains $\Omega_n$.  As an illustration, we generalize below, with a novel proof, the angle inequality of K\"{u}hnau \cite{Kuhn88}, Theorem \ref{thm:kuhnau}. One should compare the approach below to the technical difficulties that appear when studying the convergence of the boundaries $\partial \Omega_n$ in order to prove that $K_{\Omega_n}$ converges to $K_{\Omega}$ SOT, explicitly carried out for instance in the excellent work of Verchota \cite{Verc84}.
\begin{theorem} \label{thm:kuhnau}
Let $\Omega \subset \R^2$ be a $C^{1,\alpha}$-smooth curvilinear polygon. Then
\begin{equation*}
|\sigma(\overline{T_\Omega})| \geq \max_{1 \leq j \leq N} \left| 1-\theta_j/\pi \right|.
\end{equation*}
\end{theorem}
\begin{proof}
Without loss of generality we may assume that $$\max_{1 \leq j \leq N} \left| 1-\theta_j/\pi \right| = |1 - \theta_1/\pi| > 0,$$  $a_1 = 0$ and that $\gamma_1'(0) > 0$. It is then clear that when $t \to \infty$, $t \Omega$ converge as sets in $\C$ to the wedge 
\begin{equation*}
W_{\theta_1} = \{ re^{i\theta} \, : \, 0 < r < \infty, \; 0 < \theta < \theta_1\}.
\end{equation*}
By Theorem \ref{thm:krushkal} and Corollary \ref{cor:symmetry} we know that $|\sigma(\overline{T_{W_{\theta_1}}})| = |1 - \theta_1/\pi|$, either $W_{\theta_1}$ or its complement being a convex domain with two corners of equal angles $\theta_1$ or $2\pi - \theta_1$. For our technical needs, note that by mapping $W_{\theta_1}$ fractionally linearly into the plane and applying the well known Markushevich-Farrell theorem on polynomial approximation, one finds that any $f \in L^2_a(W_{\theta_1})$ can be approximated by functions $h$ which are holomorphic on $\C^* \setminus \{z_0\}$ and square integrable in a neighborhood of $\infty$, where $ \C^* = \C \cup \{ \infty \}$ and $z_0 \in \C$ is some fixed point outside $\overline{W_{\theta_1}}$. Note also that $T_{t\Omega}$ is unitarily equivalent to $T_{\Omega}$ for every $t > 0$. 

Denote by $T : L^2(\C) \to L^2(\C)$ the Beurling transform, the unitary map of $L^2(\C)$ given by
\begin{equation*}
T f (z) = \pv \frac{1}{\pi} \int_\C \frac{f(\zeta)}{(\bar{\zeta}-\bar{z})^2} \, dA(\zeta), \quad f\in L^2(\C), \, z \in \C.
\end{equation*}
For $h \in L^2_a(W_{\theta_1})$, holomorphic on  $\C^* \setminus \{z_0\}$ as described above, we have
\begin{equation*}
\|T_{t\Omega}h - T_{W_{\theta_1} }h \|_{L^2(t\Omega \cap W_{\theta_1})} \leq \|T ( \chi_{t\Omega \setminus W_{\theta_1} } h )\|_{L^2(\C)} + \|T ( \chi_{W_{\theta_1} \setminus t\Omega } h )\|_{L^2(\C)} \to 0
\end{equation*}
as $t \to \infty$, since $\chi_{t\Omega \setminus W_{\theta_1} } h \to 0$ and $\chi_{W_{\theta_1} \setminus t\Omega } h \to 0$ in $L^2(\C)$. Here $\chi_R$ denotes the characteristic function of the set $R$. For similar reasons,
\begin{equation*}
\|T_{t\Omega}h \|_{L^2(t\Omega \setminus W_{\theta_1})} \to 0, \quad \|T_{W_{\theta_1} }h \|_{L^2(W_{\theta_1} \setminus t\Omega)} \to 0, \quad t \to \infty.
\end{equation*}
In view of \eqref{eq:numrange} we may pick $h$ such that
\begin{equation*}
\int_{W_{\theta_1}} h T_{W_{\theta_1}}h \, dA \approx |\sigma(\overline{T_{W_{\theta_1}}})|.
\end{equation*}
It then follows from the above computations that for $t$ big enough,
\begin{equation*}
\int_{t\Omega} h T_{t\Omega}h \, dA \approx |\sigma(\overline{T_{W_{\theta_1}}})|,
\end{equation*}
demonstrating that
\begin{equation*}
|\sigma(\overline{T_\Omega})| = |\sigma(\overline{T_{t\Omega}})| \geq |\sigma(\overline{T_{W_{\theta_1}}})| = |1 - \theta_1/\pi|,
\end{equation*}
and hence completing the proof.
\end{proof}

We now turn to proving the bounds stated in Theorem \ref{thm:essbound} on the essential spectrum of $\overline{T_\Omega}$, and hence of $K : H^{1/2} \to H^{1/2}$, for certain curvilinear polygons. The two lemmas below show that the essential spectrum of $\overline{T_\Omega}$ is invariant under corner preserving conformal maps.
\begin{lemma} \label{lem:cornerlemma}
Let $\Omega_1$ and $\Omega_2$ be $C^{1,\alpha}$-smooth curvilinear polygons with vertices $(a_j)_{j=1}^N$ and $(b_j)_{j=1}^N$ respectively, such that the angle $0 < \theta_j < \pi$ made by $ \partial \Omega_1$ at $a_j$ is equal to the angle made by $\partial \Omega_2$ at $b_j$, for all $1 \leq j \leq N$. Suppose that $\phi : \Omega_1 \to \Omega_2$ is a biconformal map. It then automatically extends to a homeomorphism of  $\overline{\Omega_1}$ onto $\overline{\Omega_2}$. Suppose that $\phi(a_j) = b_j$ for all $j$. Then $\phi \in C^{1,\alpha}(\overline{\Omega})$ and $\phi'(z) \neq 0$ for $z\in \overline{\Omega}$.
\end{lemma}
\begin{proof}
By the classical Carath\'{e}odory theorem, $\phi$ extends to a homeomorphism $\phi : \overline{\Omega_1} \to \overline{\Omega_2}$. By our assumptions, we may factor $\psi = \psi_2 \circ \psi_1$, where $\psi_1 : \Omega_1 \to \D$ and $\psi_2 : \D \to \Omega_2$ are biconformal maps such that $\psi_1(a_j) = \psi_2^{-1}(b_j)$. 

Fix for the moment a vertex $a_j$ of $\partial \Omega_1$ and let $D_\varepsilon(\psi_1(a_j))$ be a sufficiently small closed disk of radius $\varepsilon$ centered at $\psi_1(a_j)$. For an appropriate branch of the power, it is easy to verify that the univalent map 
\begin{equation*}
\zeta(w) = (\psi_1^{-1}(w)-a_j)^{\pi/\theta_j}
\end{equation*}
on $\D$ then maps $\partial \D \cap D_\varepsilon(\psi_1(a_j))$ onto a $C^{1,\alpha}$-smooth curve. By a local version of the Kellog-Warschawski theorem, see Pommerenke \cite{Pomm02}, it follows that there exists an $\varepsilon$ such that $\zeta \in C^{1,\alpha}(\overline{\D} \cap D_\varepsilon(\psi_1(a_j)))$ and $\zeta'$ is zero-free close to $\psi_1(a_j)$. It follows that $\zeta^{-1}$ also is $C^{1,\alpha}$-smooth in some relatively open neighborhood $U_j = \zeta(\overline{\D}) \cap D_{\varepsilon'}(0)$ of $0$, 
\begin{equation}\label{eq:straighttoD}
\zeta^{-1}(z) = \psi_1(z^{\theta_j/\pi}+a_j) \in C^{1,\alpha}(U_j).
\end{equation} 
It follows that 
\begin{equation}\psi_1(z^{\theta_j/\pi}+a_j) = \psi_1(a_j) + zh_j(z),
\end{equation}
where $h_j \in C^{0,\alpha}(U_j)$ and $h_j$ is zero-free. This leads to 
\begin{equation} \label{eq:omega1toD}
\psi_1(z) = \psi_1(a_j) + (z-a_j)^{\pi/\theta_j}g_j(z)
\end{equation}
holding in a neighborhood $\overline{\Omega}_1 \cap D_{\varepsilon''}(a_j)$, where $g_j$ is $C^{0,\alpha}$-smooth, $g_j(a_j) \neq 0$. Similarly, by differentiating \eqref{eq:straighttoD} we obtain in a relatively open neighborhood of $a_j$ that
\begin{equation} \label{eq:omega1toDder}
(z-a_j)^{1-\pi/\theta_j}\psi'_1(z) = e_j(z),
\end{equation} 
$e_j$ being $C^{0,\alpha}$-smooth and $e_j(a_j) \neq 0$. Equations \eqref{eq:omega1toD} and \eqref{eq:omega1toDder} express sharpened statements of (Pommerenke \cite{Pomm92}, Theorem 3.9).

By arguing in the same manner for the map $\psi_2$ we find in a vicinity of $\psi_1(a_j)$ that
\begin{equation} \label{eq:DtoOmega2}
(w-\psi_1(a_j))^{1-\theta_j/\pi}\psi_2' (w) = f_j(w),
\end{equation}
where $f_j$ is $C^{0,\alpha}$-smooth and $f_j(\psi_1(a_j)) \neq 0$.

In combining \eqref{eq:omega1toD}, \eqref{eq:omega1toDder} and \eqref{eq:DtoOmega2} we discover that 
\begin{equation*}
\phi'(z) = \lambda e_j(z) g_j(z)^{\theta_j/\pi - 1} f_j(\psi_1(z)),
\end{equation*}
where $\lambda$ is some unimodular constant. We deduce that $\phi$ is $C^{1,\alpha}$-smooth in a relatively open neighborhood of each vertex $a_j$ in $\overline{\Omega_1}$ and $\phi'(a_j) \neq 0$. The validity of the same statement for each non-vertex point of $\partial \Omega_1$ follows immediately by the local Kellog-Warschawski theorem.
\end{proof}

\begin{lemma} \label{lem:essequiv}
Let $\Omega_1$ and $\Omega_2$ be $C^{1,\alpha}$-smooth curvilinear polygons, and let $\phi : \Omega_1 \to \Omega_2$ be a biconformal map that extends to a homeomorphism $\phi : \overline{\Omega_1} \to \overline{\Omega_2}$ such that $\phi \in C^{1,b}(\overline{\Omega_1})$ for some $0 < b < 1$ and $\phi'(w) \neq 0$ for $w \in \overline{\Omega_1}$. Then $\overline{T_{\Omega_2}}$ is unitarily equivalent to $\overline{T_{\Omega_1}} + K$ for a compact operator $K : L^2_a(\Omega_1) \to L^2_a(\Omega_1)$, linear over the field of reals.
\end{lemma}
\begin{proof}
Let the unitary map $U : L^2_a(\Omega_2) \to L^2_a(\Omega_1)$ be given by $$Uf(w) = \phi'(w) f(\phi(w)).$$ Similarly, define $\overline{U} : \overline{L^2_a(\Omega_2)} \to \overline{L^2_a(\Omega_1)}$ by $\overline{U}g(w) = \overline{\phi'(w)} g(\phi(w))$. Stokes' theorem and the change of variables $\zeta = \phi(\eta)$ gives for a polynomial $p$ that 
\begin{align*}
T_{\Omega_2}p (z) &= \lim_{\varepsilon \to 0} \frac{i}{2\pi} \left[ \int_{\partial \Omega_2} \frac{p(\zeta)}{\bar{\zeta}-\bar{z}} \, d\zeta - \frac{1}{\varepsilon^2} \int_{|\zeta-z|=\varepsilon} p(\zeta)(\zeta-z) \, d\zeta \right] \\ &= \frac{i}{2\pi} \int_{\partial \Omega_2} \frac{p(\zeta)}{\bar{\zeta}-\bar{z}} \, d\zeta = \frac{i}{2\pi} \int_{\partial \Omega_1} \frac{Up(\eta)}{\overline{\phi(\eta)}-\overline{z}} \, d\eta.
\end{align*}
Hence 
\begin{equation*}
\overline{U}T_{\Omega_2}p(w) = \frac{i}{2\pi} \int_{\partial \Omega_1} \frac{Up(\eta) \overline{\phi'(w)}}{\overline{\phi(\eta)}-\overline{\phi(w)}} \, d\eta,
\end{equation*}
and therefore the difference satisfies
\begin{equation*}
(\overline{U}T_{\Omega_2} - T_{\Omega_1}U)p(w) = \frac{i}{2\pi} \int_{\partial \Omega_1} Up(\eta) \overline{K(\eta,w)} d\eta,
\end{equation*}
where
\begin{equation*}
K(\eta,w) = \frac{\phi'(w)}{\phi(\eta)-\phi(w)} - \frac{1}{\eta - w}.
\end{equation*}
Since $\phi \in C^{1,b}(\overline{\Omega_1})$ and $|\phi'| \geq A > 0$ it is straightforward to verify that
\begin{equation*}
|K(\eta,w)| \lesssim |\eta-w|^{b-1}, \quad |\partial_\eta K(\eta,w)| \lesssim |\eta-w|^{b-2}, \quad n,w \in \Omega_2, n\neq w.
\end{equation*} 
Hence, we may apply Stokes' theorem again to obtain 
\begin{equation*}
(\overline{U}T_{\Omega_2} - T_{\Omega_1}U)p(w) = \pv \frac{1}{\pi}\int_{\Omega_1} Up(\eta) \overline{ \partial_\eta K(\eta,w) } \, dA(\eta).
\end{equation*}
That is, the difference $\overline{U}T_{\Omega_2} - T_{\Omega_1}U$ is represented by a weakly singular kernel that is continuous off the diagonal, hence well known to be a compact operator (see e.g. \cite{Cobo90}), completing the proof.
\end{proof}
Lemmas \ref{lem:cornerlemma} and \ref{lem:essequiv} combine in an obvious way. To prove Theorem \ref{thm:essbound} we require one more lemma, showing that the domains we consider can be mapped conformally onto domains in $\mathcal{V}$ in a way that preserves vertices.
\begin{lemma}
Let $\Omega$ be a $C^{1,\alpha}$-smooth curvilinear polygon with $0 < \theta_j < \pi$ for $1 \leq j \leq N$ such that its angles satisfy $$\sum_{j=1}^{N-1}(\pi-\theta_j) + \pi + \theta_N \leq 2\pi,$$
possibly after a cyclic permutation of the vertex labels. Then there exists a curvilinear polygon $V \in \mathcal{V}$ and a biconformal map $\phi : \Omega \to V$ that extends to a $C^{1,\alpha}(\overline{\Omega})$-smooth homemorphism of $\overline{\Omega}$ onto $\overline{V}$, such that $\phi'(z) \neq 0$ for $z\in \overline{\Omega}$.
\end{lemma}
\begin{proof}
Let $\psi_1 : \Omega \to \D$ be a biconformal map. By the classical Carath\'{e}odory theorem, it extends to a homeomorphism $\psi_1 : \overline{\Omega} \to \overline{\D}$. Let $\psi_2$ be the map defined on $\D$ by
\begin{equation*}
\psi_2(w) = \int_0^w (z-\psi_1(a_N))^{-(1+\theta_N/\pi)}\Pi_{j=1}^{N-1}(z- \psi_1(a_j))^{\theta_j/\pi -1} \, dz,
\end{equation*}
inspired by the Schwarz-Christoffel transformations. We will soon verify that  $\psi_2$ is univalent on $\D$ with convex range. Based on this, one easily checks that $\partial \psi_2(\D) \subset \C^* = \C \cup \{\infty\}$ is an unbounded piecewise real analytic Jordan curve having exactly $N$ vertices, such that each vertex is given by $\psi_2(\psi_1(a_j))$ for some $j$ and the interior angle made there is equal to $\theta_j$. Note in particular that $\psi_2(\psi_1(a_N)) = \infty$ and that $\psi_2(\D) \in \mathcal{V}$. To obtain a bounded curvilinear polygon $V \in \mathcal{V}$, let $V = \psi_3(\psi_2(\D))$, where $\psi_3(z) = \frac{1}{z-z_0}$ for a point $z_0 \notin \overline{\psi_2(\D)}$. To finish the lemma, let $\phi = \psi_3 \circ \psi_2 \circ \psi_1$ and apply Lemma \ref{lem:cornerlemma}.

To check the convexity of $\psi_2$, we mimic Johnston \cite{John83} and note that
\begin{multline*}
\Re \left( 1 + w\frac{\psi_2''(w)}{\psi_2'(w)} \right) = \\ 1 - \Re\left( (1+\theta_N/\pi)\frac{w}{w-\psi_1(a_N)} + \sum_{j=1}^{N-1} (1-\theta_j/\pi)\frac{w}{w-\psi_1(a_j)}\right).
\end{multline*}
Since $w \mapsto \xi = \frac{w}{w-z}$ maps $\D$ biconformally onto $\Re \xi < 1/2$ for $z \in \partial \D$, it follows by our assumptions on the angles that $\Re \left( 1 + w\frac{\psi_2''(w)}{\psi_2'(w)} \right) > 0$ for $w \in \D$. It is well known that this condition implies that $\phi_2$ is univalent with convex range.
\end{proof}

In combining the previous three lemmas we finally obtain our desired result.
\begin{theorem} \label{thm:essbound}
Let $\Omega$ be a $C^{1,\alpha}$-smooth curvilinear polygon with $0 < \theta_j < \pi$ for $1 \leq j \leq N$ such that its angles satisfy $$\sum_{j=1}^{N-1}(\pi-\theta_j) + \pi + \theta_N \leq 2\pi,$$
possibly after a cyclic permutation of the vertex labels. Then
\begin{equation*}
|\sigma_\textrm{ess}(K)| = |\sigma_\textrm{ess} (\overline{T_\Omega})| \leq \max_{1 \leq j \leq N} \left( 1-\theta_j/\pi \right),
\end{equation*}
where the Neumann-Poincar\'{e} operator $K$ is acting on $H^{1/2}(\partial \Omega)$.
\end{theorem}

As a finishing remark, we conjecture, partly based on numerical evidence, that $|\sigma_\textrm{ess} (\overline{T_\Omega})|$ only depends on the interior angles of a curvilinear polygon $\Omega$, and not on its shape. On the other hand, for the full spectrum, Werner \cite{Wern97} proves that $|\sigma (\overline{T_R})| = 1/2$ for rectangles $R$ sufficiently close to a square, but that $|\sigma (\overline{T_R})| > 1/2$ when the ratio between the side lengths of $R$ exceeds $2.76$, suggesting the formation of eigenvalues outside the essential spectrum as a rectangle is elongated.

\section*{Acknowledgements}
We thank Professors Alexandru Aleman, Johan Helsing and Jan-Fredrik Olsen for very helpful discussions. The work of the first author was partially supported by The Royal Physiographic Society in Lund. The work of the second author was partially supported by the National Science Foundation Grant DMS-10-01071.

\end{document}